\documentclass[12pt]{amsart}
\usepackage{amssymb, amstext, amscd, amsmath}
\usepackage{mathtools, xypic, color, rotating}

\makeatletter
\def\@cite#1#2{{\m@th\upshape\bfseries%
[{#1\if@tempswa{\m@th\upshape\mdseries, #2}\fi}]}}
\makeatother
\theoremstyle{plain}
\newtheorem{thm}{Theorem}[section]
\newtheorem{cor}[thm]{Corollary}
\newtheorem{prop}[thm]{Proposition}
\newtheorem{lem}[thm]{Lemma}
\newtheorem{keylem}[thm]{Key Lemma}
\theoremstyle{definition}
\newtheorem{rem}[thm]{Remark}

\renewcommand{\labelenumi}{(\roman{enumi}) }
\mathtoolsset{centercolon}
\newcommand{\bC}{{\mathbb{C}}}
\newcommand{\bT}{{\mathbb{T}}}
\newcommand{\bZ}{{\mathbb{Z}}}
\newcommand{\B}{{\mathcal{B}}}
\renewcommand{\H}{{\mathcal{H}}}
\newcommand{\M}{{\mathcal{M}}}
\renewcommand{\P}{{\mathcal{P}}}
\newcommand{\R}{{\mathcal{R}}}

\renewcommand{\phi}{\varphi}
\def\si{\sigma}
\def\al{\alpha}
\def\be{\beta}

\def\ga{\gamma}
\newcommand\wpi{\widetilde{\pi}}

\newcommand{\Bv}{{\mathbf{v}}}
\newcommand{\Bw}{{\mathbf{w}}}
\newcommand{\FORAL}{\text{ for all }}

\newcommand{\qand}{\quad\text{and}\quad}
\newcommand{\qfor}{\quad\text{for}\ }
\newcommand{\qforal}{\quad\text{for all}\ }

\newcommand{\ltwo}{\ell^2}
\newcommand{\ol}{\overline}
\newcommand{\ad}{\operatorname{ad}}
\newcommand{\ann}{\operatorname{ann}}
\newcommand{\diag}{\operatorname{diag}}
\newcommand{\End}{\operatorname{End}}
\newcommand{\id}{{\operatorname{id}}}
\newcommand{\scp}[2]{#1 \times_{#2} \bZ_+} 

\begin{document}

\title[Conjugate dynamical systems]{Conjugate dynamical systems\\ on C*-algebras}
\thanks{}

\author[K.R. Davidson]{Kenneth R. Davidson}
\address{Pure Math.\ Dept.\\U. Waterloo\\Waterloo, ON\;
N2L--3G1\\CANADA}
\email{krdavids@uwaterloo.ca}

\author[E.T.A. Kakariadis]{Evgenios T.A. Kakariadis}
\address{Pure Math.\ Dept.\\U. Waterloo\\Waterloo, ON\;
N2L--3G1\\CANADA}
\email{ekakaria@uwaterloo.ca}

\begin{abstract}
Let $(A, \alpha)$ and $(B, \beta)$ be C*-dynamical systems
where $\alpha$ and $\beta$ are arbitrary $*$-endomorphisms.
When $\al$ is injective or surjective,
we show that the semicrossed products $\scp{A}{\alpha}$ and
$\scp{B}{\beta}$ are isometrically isomorphic if and only if
$(A, \alpha)$ and $(B, \beta)$ are outer conjugate.
This conclusion also holds in various other cases as well.
\end{abstract}

\thanks{First author partially supported by an NSERC grant.}
\thanks{Second author partially supported by the
Fields Institute for Research in the Mathematical Sciences}
\date{}
\maketitle

\section{introduction} \label{S:intro}

If $\alpha$ is a $*$-endomorphism of a C*-algebra $A$, the semicrossed product
$\scp A \alpha$ is an operator algebra which encodes the dynamics in the sense that
it is the universal operator algebra for covariant representations of $(A,\alpha)$.
We wish to determine to what extent the dynamical system can be recovered
from the semicrossed product.
It is easy to see that if two systems are outer conjugate, then the semicrossed products
are completely isometrically isomorphic.
We establish the converse when $\al$ is injective, and in various other cases,
by showing that if two semicrossed products are isometrically isomorphic,
then the dynamical systems are outer conjugate.

The use of nonself-adjoint operator algebras to encode C*-dynamics goes
back to Arveson \cite{Arv,ArvJ} where ergodic actions on a space $X$
were encoded in a concrete operator algebra.
Under certain hypotheses, two such algebras were shown to be
isomorphic if and only if the two (commutative) dynamical systems are conjugate.
Peters \cite{Peters} established that under much more general conditions,
one obtains a universal operator algebra encoding the covariance relations.
He extended the isomorphism results for the commutative case,
although there were still conditions on the fixed point sets.
This was further developed by Hadwin and Hoover \cite{HadH}.
Finally in \cite{DK}, Davidson and Katsoulis showed that (in the commutative case),
algebraic isomorphism of semicrossed products is a complete invariant for
conjugacy. See \cite{DKsurvey} for a survey of these results.

The use of non-selfadjoint operator algebras is not just an artifact of convenience.
They arise naturally when one tries to model the covariance relations.
Moreover, the C*-algebra crossed product loses information.
Even when $\alpha$ is an automorphism, so that the usual C*-algebra crossed product
$A\rtimes_\al \mathbb{Z}$ is available, 
this algebra is not a complete invariant for the system. 
Hoare and Parry \cite{HoaPar66} give an example of an automorphic, commutative 
dynamical system $(A,\al)$ such that $\al$ and $\al^{-1}$ are not conjugate.
However $A\rtimes_\al \mathbb{Z}$ and $A\rtimes_{\al^{-1}} \mathbb{Z}$ are always 
$*$-isomorphic.

In the non-automorphic case, there is a variety of C*-algebra crossed products by
endomorphisms \cite{Pas,Stac,Exel03}, 
introduced as possible generalizations of the crossed product.
As in the automorphic case, these algebras do not generally allow recovery of the
dynamics from the algebra.

Likewise there are several possible choices for the semicrossed product.
See \cite{KakKat11,Kak11} for a discussion.
When the $*$-endomorphism is injective, the various choices produce 
the `same' operator algebra.
The conclusion is basically that the original choice made by Peters is
still the best option as one always obtains an isometric copy of the C*-algebra
in this semicrossed product. 

There is not much literature showing that the semicrossed product is an invariant
for a non-commutative dynamical system, even when $\al$ is an automorphism. 
One has to replace conjugacy by outer conjugacy, because outer conjugate systems 
yield completely isometric semicrossed products.  See section~\ref{S:back}. 
The first result in this direction was given by Muhly and Solel in \cite[Theorem 4.1]{MS00} 
where they show that if $A$ is a separable C*-algebra 
and $\al$ is an automorphism with a full Connes spectrum, 
then the semicrossed product is an isometric isomorphism invariant.
Davidson and Katsoulis \cite{DKsimple} extend this result for separable, 
simple C*-algebras, also when $\al$ is an automorphism, by using 
nest representations and
a result of Kishimoto \cite{Kish81} on universally weakly inner automorphisms.

In this paper, we show that this technology is unnecessary. We are able to show that
isometric isomorphism of semicrossed products of automorphic dynamical systems is a complete invariant for
conjugacy, without making any extra assumption on the C*-algebras.

In fact we do more than that. We are able to prove the same for arbitrary unital C*-algebras provided that we make the common assumption
that the endomorphisms are injective or surjective.
The general case, in which the endomorphisms are not surjective and have kernel, is delicate;
and we have not resolved this case completely.
To our knowledge, these are the first significant results in the literature for this situation.

The main results established here are:

\begin{thm} \label{T:main}
Let $(A,\al)$ and $(B,\be)$ be unital C*-algebra dynamical systems.
Suppose that $\al$ is either injective or surjective.
Then $A\times_\al \bZ_+$ and $B\times_\be \bZ_+$ are isometrically isomorphic
if and only if $(A,\al)$ and $(B,\be)$ are outer conjugate.
\end{thm}

We conjecture that no hypotheses on $\al$ are required. We gather considerable
evidence towards this view, and establish the result in a number of other cases.

\begin{thm} \label{T:main2}
Let $(A,\al)$ and $(B,\be)$ be unital C*-algebra dynamical systems.
Suppose that one of the following holds:
\begin{itemize}
\item $A$ has trivial centre, $Z(A) = \bC 1$.  $($e.g. $A$ simple.$)$
\item $A$ is abelian.
\item $A$ is finite $($no proper isometries$)$.
\item $\al(A)'$ is finite.
\item $\al(R_\al) = R_\al$, where $R_\al = \ol{\bigcup_{k\ge1}\ker(\al^k)}$.
\item $\al(\ann(R_\al)) \subseteq \ann(R_\al)$.
\end{itemize}
Then $A\times_\al \bZ_+$ and $B\times_\be \bZ_+$ are isometrically isomorphic
if and only if $(A,\al)$ and $(B,\be)$ are outer conjugate.
\end{thm}

\section{Background} \label{S:back}

A C*-dynamical system will be a unital C*-algebra $A$ together with a unital
$*$-endomorphism $\alpha \in \End(A)$.
A covariant representation of $(A,\alpha)$ is a pair $(\pi,V)$ where
$\pi$ is a $*$-representation of $A$ on a Hilbert space $\H$ and
$V\in \B(\H)$ is a contraction satisfying the covariance relation
\[ \pi(a) V = V \pi\al(a) \qforal a \in A .\]

Let $\P$ be the space of formal polynomials of the form
$\sum_{k=0}^n \Bv^k a_k$ where $a_k \in A$ and $n\ge0$
with the natural vector space structure, and a multiplication given by
the rule
\[ a\Bv = \Bv \alpha(a) \qfor a \in A .\]
Then it is evident that a covariant pair $(\pi,V)$ yields a representation of $\P$ by
\[ \pi\times V\big (\sum_{k=0}^n \Bv^k a_k \big) = \sum_{k=0}^nV^k \pi(a_k) .\]
A norm is defined by
\[ \|p\| = \sup_{(\pi,V) \ \text{covariant}} \|  \pi\times V(p) \| .\]
This quantity is finite because
\[ \|p\| \le \|p\|_1 := \sum_{k=0}^n \|a_k\| .\]
This likewise determines norms on $\M_n(\P)$ for  each $n \ge1$.

The \textit{semicrossed product} $\scp{A}{\alpha}$ is the universal operator algebra
for this family of representations.
That is, it is the operator space completion of $\P$ in this family of matrix norms.
The characterizing property is that every covariant pair $(\pi,V)$ yields a
completely contractive representation $\pi\times V$ of $\scp{A}{\alpha}$
extending the representation on $\P$.

Covariant representations always exist.
Indeed, let $\pi$ be any representation of $A$ on $\H$.
Let $\H^{(\infty)} = \H \otimes \ltwo$ be the infinite ampliation of $\H$.
Define a representation $\wpi$ on $\H^{(\infty)}$ by
\[ \wpi(a) =  \diag ( \pi\al^k(a): k\ge0 ) . \]
Let $S$ be the unilateral shift on $\ltwo$ and set $V = I \otimes S$.
Then it is straightforward to check that $(\wpi,V)$ is a covariant representation,
known as the \textit{orbit representation induced by $\pi$}.
Peters \cite{Peters} shows that if $\pi$ is a faithful representation of $A$,
then $\wpi \times V$ is an isometric representation
of $\scp{A}{\alpha}$.
A different proof is contained in \cite{Kak11} based on a gauge
invariance uniqueness theorem, that additionally provides that $\wpi\times V$ is completely isometric.

One could also define covariant representations that insist that $V$ be an isometry.
By a dilation result of Muhly and Solel \cite{MS06}, every contractive covariant pair
dilates to an isometric covariant pair. Hence there is no difference in the universal
operator algebra obtained.
Moreover, the element $\Bv$ in $\scp A \alpha$ is an isometry in its C*-envelope.

If $(\pi,V)$ is a covariant representation and $z \in \bT = \{z : |z|=1\}$,
then $(\pi,zV)$ is also a covariant representation.
It follows from the universal property that there is a completely isometric
automorphism $\gamma_z$ of $\scp A \alpha$ such that $\gamma_z(a)=a$ for
$a \in A$ and $\gamma_z(\Bv)=z\Bv$.
The map $z \to \gamma_z$ is a group homomorphism which is point-norm
continuous.
This enables us to define Fourier coefficients by
\[ \Bv^n E_n(X) =  \int_{\bT} \gamma_z(X) \bar z^n dm \qfor n \ge 0, \]
where $m$ is normalized Lebesgue measure on $\bT$.
These are evidently completely contractive maps of $\scp A \alpha$ into $A$.
In particular, $E_0$ is an expectation on $\scp A \alpha$ onto $A$
(i.e. a completely contractive idempotent map) which is also a homomorphism.
The standard proof of F\'ejer's Theorem shows that
\[ \Sigma_n(X) = \sum_{k=0}^{n-1} \big( 1 - \tfrac k n \big) \Bv^k E_k(X) \]
are completely contractive maps, and that $\Sigma_n(X)$ converges to $X$ in norm
for all $X \in \scp A \alpha$.
In particular, the map taking $X$ to its Fourier series is injective.
\medskip

Two dynamical systems $(A,\al)$ and $(B,\be)$ are said to be \emph{outer conjugate}
if there is a $*$-isomorphism $\ga$ of $A$ onto $B$ and a unitary $v\in A$
such that
\[
 \al(a) = \ad_v \ga^{-1}\be\ga(a) = v(\ga^{-1}\be\ga(a))v^* \qforal a\in A .
\]
Setting $w = \gamma(v)^*$, it is easy to deduce that
\[
 \be(b) = \ad_w \ga\al\ga^{-1}(b) = w(\ga\al\ga^{-1}(b))w^* \qforal b\in B .
\]

In this case the semicrossed products $A\times_\al \bZ_+$ and
$B\times_\be \bZ_+$ are completely isometrically isomorphic isomorphic.
To see this, write elements of $B\times_\be \bZ_+$ as Fourier series
in a variable $\Bw$ and coefficients in $B$.
Observe that for $a \in A$,
\[
 \ga(a) (\Bw w) = \Bw \beta\ga(a) w = \Bw w (\ga\al\ga^{-1}\ga(a)) w^* w
 = (\Bw w) \ga\al(a).
\]
Thus $(\gamma, \Bw w)$ is a covariant representation of $(A,\alpha)$.
Therefore there is a completely contractive homomorphism $\phi$ of
$A\times_\al \bZ_+$ into $B\times_\be \bZ_+$
so that $\phi|_A= \ga$ and $\phi(\Bv)=\Bw w$.
Similarly there is a map $\psi$ of $B\times_\be \bZ_+$ into $A\times_\al \bZ_+$
such that $\psi|_B = \ga^{-1}$ and $\psi(\Bw) = \Bv v$.
It is easy to check that $\psi = \phi^{-1}$, and thus $\phi$ is a
unital completely isometric isomorphism.

In this paper, we are concerned with the converse.
Suppose that $A\times_\al \bZ_+$ and $B\times_\be \bZ_+$ are (completely)
isometrically isomorphic. Are $(A,\al)$ and $(B,\be)$ outer conjugate?
A positive answer for the injective and surjective cases is the content of Theorem~\ref{T:main}.
Theorem~\ref{T:main2} includes extensions to many situations where
$\al$ is not injective. This requires a careful analysis of the sequence of ideals $\ker(\al^k)$
and their closed union, $R_\al$.

\section{Algebraic Manipulations} \label{S:alg}

Suppose that $\phi$ is an isomorphism of $A\times_\al \bZ_+$ onto $B\times_\be \bZ_+$.
The universal property shows \cite{Kak11} that $\phi$ is completely isometric
if and only if it is isometric.
We assume from now on that $\phi$ is a fixed (completely) isometric isomorphism.
A standard argument, using the fact that $A$ is the largest C*-algebra contained in
$\scp A \alpha$ and that an isometric isomorphism of a C*-algebra is a
$*$-isomorphism, shows that $\ga = \phi|_A$ is a $*$-isomorphism of $A$ onto $B$.
Conversely, if $\phi|_A$ is a $*$-isomorphism of $A$ onto $B$
and $\|\phi(\Bv)\| = 1 = \|\phi^{-1}(\Bw)\|$, then $\phi$ is completely isometric.

In this section, we develop a variety of algebraic results that are quite
ring theoretic in nature. The only use of the isometric hypothesis is
the consequence mentioned in the previous paragraph that the restriction
to the C*-algebra $A$ is a $*$-isomorphism onto $B$.

Consider the image of the generator $\Bv$ of $A\times_\al \bZ_+$ under $\phi$.
Using the Fourier decomposition, we may write
\[
 \phi(\Bv) = b_0 + \Bw b_1 + \Bw^2 Y   \qand
 \phi^{-1}(\Bw) = a_0 + \Bv a_1 + \Bv^2 X ,
\]
where $Y \in \scp B \be$ and $X \in \scp A \al$.
Note that the elements of the form $\Bw^2 Y$ form a closed ideal of $\scp B \be$
consisting of those elements $Z$ such that $E_0(Z)=E_1(Z)=0$.

\begin{lem} \label{L:identities} 
For all $a\in A$ we obtain:
\begin{enumerate}
\item $\ga(a)\, b_0 \ =\, b_0\, \ga\al(a)
 \qand\ \,  \be\ga(a)\, b_1 = b_1\, \ga\al(a)$.

\item $b_0b_0^*$ lies in the centre $Z(B)$;
$b_0b_0^*$ and $b_1^*b_1$ commute with $\ga\al(A)$;
and $b_1b_1^*$ commutes with $\be(B)$.
\end{enumerate}
The symmetrical relations hold for $a_0$ and $a_1$. Moreover,
\begin{enumerate}
\item[(iii)] $a_1 \ga^{-1}(b_1)$ commutes with $\al(A)$.
\end{enumerate}
\end{lem}

\begin{proof}
Rewrite the covariance relations for $(A,\al)$ under $\phi$, for $a \in A$, as
\begin{align*}
 &\hspace{9mm} \phi(a\Bv) \hspace{2mm} = \ga(a) b_0 + \ga(a) \Bw b_1 + \ga(a) \Bw^2 Y \\
 &\hspace{22mm} = \ga(a) b_0 + \Bw \be\ga(a) b_1 + \Bw^2 \be^2\ga(a)  Y\\
 &= \phi(\Bv \al(a)) = b_0 \ga\al(a) + \Bw b_1 \ga\al(a) +  \Bw^2 Y \ga\al(a) .
\end{align*}
Therefore, equating the Fourier coefficients of the two expressions, we obtain that
\[
 \ga(a)\, b_0 = b_0\, \ga\al(a) \qand
 \be\ga(a)\, b_1 = b_1\, \ga\al(a)
 \qforal a \in A .
\]
This yields (i).
Since $\ga$ is an isomorphism of $A$ onto $B$, we obtain the useful variant
\[
 bb_0 = b_0 \, (\ga\al\ga^{-1})(b) \qand
 \be(b) b_1 = b_1\, (\ga\al\ga^{-1})(b)  \qforal b \in B.
\]
Since $A$ is a C*-algebra, conjugation of the first identity yields
\[ b_0^* \ga(a) = \ga\al(a) b_0^* \qforal a \in A.\]
For $b\in B$, let $a = \ga^{-1}(b)$. We obtain
\begin{align*}
 (b_0b_0^*) b &= b_0b_0^* \ga(a) = b_0 \ga\al(a) b_0^* 
 = \ga(a)  b_0 b_0^* = b (b_0 b_0^*).
\end{align*}
This establishes the first part of (ii).
The other statements are established similarly.
By symmetry, we get the analogous relations for $a_0$ and $a_1$.

Statement (iii) is also similar. Let $a \in A$. Then
\begin{align*}
a_1 \ga^{-1}(b_1) \al(a) &= a_1 \ga^{-1}\big( b_1 \ga\al(a) \big) \\&
= a_1 \ga^{-1}\big( \be\ga(a) b_1 \big)
= \al(a) a_1 \ga^{-1}(b_1). \qedhere
\end{align*}
\end{proof}

\begin{lem} \label{L:rightinv}
In the notation above, $a_1$ and $b_1$ are right invertible.
\end{lem}

\begin{proof}
It suffices to prove the result for $b_1$. Observe that since $\phi$ is surjective,
there is a polynomial $p = \sum_{k=0}^n \Bv^k x_k$, with $x_k\in A$,
so that $\|\phi(p) - \Bw \| < 1/2$. Compute
\begin{align*}
 E_1(\phi(p)) &= E_1 \sum_{k=0}^n \phi(\Bv)^k \ga(x_k) \\
 &= \sum_{k=0}^n E_1 \big((b_0 + \Bw b_1 + \Bw^2 Y)^k \ga(x_k)\big) \\
 &= \sum_{k=0}^n \sum_{i=0}^{k-1} E_1 \big(b_0^i \Bw b_1 b_0^{k-1-i} \ga(x_k)\big) \\
 &= \sum_{k=0}^n \sum_{i=0}^{k-1} \be(b_0^i) b_1 b_0^{k-1-i} \ga(x_k) \\
 &= b_1 \sum_{k=0}^n \sum_{i=0}^{k-1} \ga\al\ga^{-1}(b_0^i)\, b_0^{k-1-i} \ga(x_k)
 = b_1 c.
\end{align*}
Therefore
\[ \| b_1c - 1\| = \| E_1\big( \phi(p) - \Bw \big) \| < 1/2 .\]
It follows that $b_1c$ is invertible, and hence $b_1$ is right invertible.
\end{proof}

\begin{cor} \label{C:ker}
If $\phi$ is an isometric isomorphism of $A\times_\al \bZ_+$ onto $B\times_\be \bZ_+$,
then $\ga(\ker\al) = \ker\be$.
\end{cor}

\begin{proof}
If $a \in\ker\al$, then
\[ \be\ga(a)b_1 = b_1 \ga\al(a) = 0 .\]
Since $b_1$ has a right inverse, $\ga(a) \in\ker\be$.
Thus $\ga(\ker\al) \subset \ker\be$.
Similarly, $\ga^{-1} (\ker\be) \subset \ker\al$. Therefore we have equality.
\end{proof}

We now establish a key general result.

\begin{keylem} \label{L:key}
If $b_1$ is invertible, then $(A,\al)$ and $(B,\be)$ are outer conjugate.
\end{keylem}

\begin{proof}
By Lemma~\ref{L:identities}(ii), $b_1^*b_1$ commutes with $\ga\al(A)$.
Thus $|b_1| = (b_1^*b_1)^{1/2}$ also commutes with $\ga\al(A)$.
Let $b_1 = w|b_1|$ be the polar decomposition.
Since $b_1$ is invertible by hypothesis, $w$ is unitary.
Compute
\begin{align*}
 \big( w\, \ga\al(a)\, w^* \big) b_1 &=
 w \, \ga\al(a) |b_1| = w|b_1|\, \ga\al(a)   \\&=
 b_1\, \ga\al(a)  = \be\ga(a)\, b_1 .
\end{align*}
Therefore $w\, \ga\al(a)\, w^* = \be\ga(a)$.  Substituting $b=\ga(a)$ yields
\[ \be(b) = (\ad_w \ga\al\ga^{-1}) (b) \qforal b \in B .\]
This establishes outer conjugacy.
\end{proof}

The following is a special case in which we obtain invertibility.

\begin{lem} \label{L:a0=0}\ 
\begin{enumerate}
\item  If $a_0=0$ or $b_0=0$, then $a_0=0=b_0$,
$b_1$ and $a_1$ are unitary, and $a_1 \ga^{-1}(b_1)=1$.

\item If $X=0$ or $Y=0$, then $X=0=Y$,
$b_1$ and $a_1$ are unitary, and $a_1 \ga^{-1}(b_1)=1$.
\end{enumerate}
Consequently, $\phi(\Bv)=\Bw b_1$, $\phi^{-1}(\Bw)=\Bv a_1$ and $\ga(a_1)=b_1^*$, in either case.
\end{lem}

\begin{proof}
By symmetry we can assume that $a_0=0$ or $Y=0$. 
Let us write $\phi^{-1}(Y) = c_0+\Bv c_1 + \Bv^2 Z$.  
First we compute
\begin{align*}
 \Bv &= \phi^{-1}(\phi(\Bv)) = \phi^{-1}(b_0+\Bw b_1 + \Bw^2Y) \\
 &= \gamma^{-1}(b_0) + (a_0+\Bv a_1 + \Bv^2 X)\gamma^{-1}(b_1) \\ &\quad+
 (a_0+\Bv a_1 + \Bv^2 X)^2 (c_0+\Bv c_1 + \Bv^2 Z) \\
 &= \big( \ga^{-1}(b_0) + a_0 \ga^{-1}(b_1) + a_0^2 c_0 \big) \\&\quad
 + \Bv \big( a_1 \ga^{-1}(b_1) + \al(a_0)a_1c_0 + a_1a_0c_0 + \al(a_0)^2c_1 \big) + \Bv^2 Z'
\end{align*}
Therefore
\begin{align*}
 0 &= \ga^{-1}(b_0) + a_0 \ga^{-1}(b_1) + a_0^2 c_0, \\
 \intertext{and}
 1 &= a_1 \ga^{-1}(b_1) + \al(a_0)a_1c_0 + a_1a_0c_0 + \al(a_0)^2c_1.
\end{align*}

If $a_0=0$, then it follows that $b_0=0$, and so $a_1 \ga^{-1}(b_1) = 1$. 
If $Y=0$, then $c_0=c_1=0$; and therefore $a_1 \ga^{-1}(b_1)=1$ again. 
In either case $\ga^{-1}(b_1)$ is left invertible. Thus $b_1$ is invertible,
and $a_1=\ga^{-1}(b_1^{-1})$ is also invertible.
Since $a_1$ and $b_1$ are  contractions, it follows that they must be unitary.

Also when $Y=0$, the identity for $\Bv= \phi^{-1}(\phi(\Bv))$ simplifies
and yields $\Bv^2X\ga^{-1}(b_1)=0$. Since $b_1$ is right invertible, $X=0$ as well. 
By symmetry, it suffices to check just these cases.

Now choose a faithful representation $\pi$ of $B$ and consider the
orbit representation $(\tilde\pi,I\otimes S)$ that it induces.
Applying this to $\phi(\Bv)$ yields a matrix operator with the first column
equal to $(b_0,b_1,b_2,\dots)^t$ where $\Bv^2 Y = \sum_{n\ge2} \Bv^n b_n$.
This column has norm at most one, and $b_1$ is unitary.
Hence $b_n = 0$ for all $n \ne 1$. 
That is, $b_0=0$ and $Y=0$, and thus  $\phi(\Bv) = \Bw b_1$.
Similarly $\phi^{-1}(\Bw) = \Bv a_1$.
\end{proof}

The following remark provides invertibility in cases when $a_0$, 
and as a consequence $X$, are not zero.

\begin{rem} \label{R:a0small}
\textit{If $\|a_0\| < \delta :=\tfrac 23{\sqrt{3}} -1 \simeq 0.1547$,
then $b_1$ and $a_1$ are invertible.}

We know that $\|a_1\| \le \|\Bw\| = 1$ and $\|Y\| \le 1 + \|b_0\|+\|b_1\| \le 3$.
Since $\phi^{-1}(Y) = c_0+\Bv c_1 + \Bv Z$, we see that $\|c_0\|\le 3$ and $\|c_1\| \le 3$.
So we can plug these estimates into the identity obtained in the
proof of the previous lemma to obtain
\begin{align*}
 \| 1 - a_1 \gamma^{-1}(b_1) \| &= \| \al(a_0)a_1c_0 + a_1a_0c_0 + \al(a_0)^2c_1\| \\&
 \le 6\|a_0\| + 3 \|a_0\|^2 < 6 \delta + 3 \delta ^2 = 1 .
\end{align*}
It follows that  $a_1\ga^{-1}(b_1)$ is invertible.
Therefore $\gamma^{-1}(b_1)$ is left invertible; and thus $b_1$ is invertible.
Moreover, if $c$ is the inverse of $a_1\ga^{-1}(b_1)$, then $b_1^{-1} = \ga(ca_1)$.
Hence $\ga^{-1}(b_1)c$ is the inverse of $a_1$.
\end{rem}

When $(A,\al)$ is not injective, we need careful analysis of the
family of kernels $\ker(\al^n)$ for $n\ge1$ and their closed union
$R_\al = \ol{ \cup_n \ker\al^n}$, known as the \textit{radical ideal} .
It is straightforward to prove that $a\in R_\al$ if and
only if $\lim_n \al^n(a)=0$.

\begin{lem} \label{L:radical}
If $A\times_\al \bZ_+$ is isometrically isomorphic to $B\times_\be \bZ_+$,
then there are right invertible elements $b_n \in B$ such that
\[
 \be^n\ga (a) \, b_n = b_n \, \ga \al^n(a) \qforal a \in A.
\]
\end{lem}

\begin{proof}
We will prove by induction that $b_n = \be^{n-1}(b_1) \cdots b_1$.
Note that $b_{n+1}=\be(b_n)b_1$.
Since $b_1$ is right invertible by Lemma~\ref{L:rightinv}, it follows
that each $b_n$ is also right invertible.

It has already been shown that
\[ \be\ga(a)\, b_1 = b_1\, \ga\al(a)  \qforal a \in A . \]
Suppose that
\[ \be^n\ga (a) \, b_n = b_n \, \ga \al^n(a) \qforal a \in A . \]
Then
\begin{align*}
 \be^{n+1}\ga(a) b_{n+1} &=\be(\be^n\ga(a)) \be(b_n)b_1 =
 \be(\be^n\ga (a) b_n) b_1 \\&=
 \be(b_n \ga \al^n(a)) b_1  = \be(b_n) \be \ga (\al^n(a)) b_1 \\&=
 \be(b_n)b_1 \ga\al (\al^n(a)) = b_{n+1} \ga\al^{n+1}(a) .
 \qedhere
\end{align*}
\end{proof}

\begin{cor}\label{C:kernels}
If $\phi$ is an isometric isomorphism of $A\times_\al \bZ_+$ onto $B\times_\be \bZ_+$,
then $\ga(\ker(\al^n)) = \ker(\be^n)$ for $n\ge1$, and $\ga(R_\al) = R_\be$.
\end{cor}

\begin{proof}
Corollary~\ref{C:ker} showed that $\ga(\ker\al) = \ker\be$.
Using Lemma~\ref{L:radical}, the same argument shows that
$\ga(\ker\al^n) = \ker\be^n$ for each $n \ge 1$.
Therefore $\ga(R_\al) = R_\be$.
\end{proof}

If $J$ is an ideal of a C*-algebra $B$, then the \textit{annihilator} of $J$ is
\[  \ann(J) = \{ b \in B : bJ = 0 \} = \{ b\in B : Jb = 0 \} .\]
It is a consequence of the fact that $J$ has a bounded approximate identity
that the left and right annihilators coincide, and hence $ \ann(J)$ is a (closed) ideal of $B$.

\begin{lem} \label{L:annihilator}
With notation as before, $a_0 \in \ann(R_\al)$ and
$b_0 \in \ann(R_\be)$.
\end{lem}

\begin{proof}
By symmetry, it suffices to establish the result for $b_0$.
Recall Lemma~\ref{L:identities}(i):
\[ \ga(a) b_0 = b_0 \ga\al(a) \qforal a \in A .\]
If $x \in \ker\be$, then $x = \ga(a)$ for some $a \in \ker\al$ by Corollary~\ref{C:kernels}.
Thus
\[ xb_0 = \ga(a) b_0 = b_0 \ga\al(a) = 0. \]
Hence $b_0 \in \ann(\ker\be)$.

Assume $b_0 \in \ann(\ker(\be^n))$.
Take $x \in \ker(\be^{n+1})$.
Since $\ga(\ker(\al^{n+1})) = \ker(\be^{n+1})$ by Corollary~\ref{C:kernels},
there is an $a \in \ker(\al^{n+1})$ so that $x = \ga(a)$.
Moreover since $\al(a) \in \ker(\al^n)$, we get
\[  xb_0 = \ga(a) b_0 = b_0 \ga\al(a) = 0 \]
by the inductive hypothesis.
Hence $b_0 \in \ann(\ker(\be^{n+1})$.

Taking the union for $n \ge 1$ yields $b_0 \in \ann(R_\be)$.
\end{proof}

\section{Main Theorem~\ref{T:main}} \label{S:main}

\subsection{The surjective case} 

First we assume that $\al$ is surjective.

\begin{thm} \label{T:surj}
Let $(A,\al)$ and $(B,\be)$ be unital C*-algebra dynamical systems.
Suppose that $A\times_\al \bZ_+$ and $B\times_\be \bZ_+$ are isometrically isomorphic.
If $\al$ is surjective, then $(A,\al)$ and $(B,\be)$ are outer conjugate.
\end{thm}

\begin{proof}
It suffices to show that $b_1$ is invertible.
If $b_1 c = 1$, then since $\ga\al\ga^{-1}$ is surjective, there is an element
$b\in B$ so that $\ga\al\ga^{-1}(b)=c$.  Therefore
\[ \be(b)b_1 = b_1 \ga\al\ga^{-1}(b) = b_1 c = 1 .\]
So $b_1$ has a left and a right inverse; and hence it is invertible.
Therefore $(A,\al)$ and $(B,\be)$ are outer conjugate by the Key Lemma~\ref{L:key}.
\end{proof}

The following corollary is an immediate consequence which we record for further use and reference.

\begin{cor}\label{C:auto}
Let $(A,\al)$ and $(B,\be)$ be unital C*-algebra dynamical systems.
Suppose that $A\times_\al \bZ_+$ and $B\times_\be \bZ_+$ are isometrically isomorphic.
If $\al$ is an automorphism, then $(A,\al)$ and $(B,\be)$ are outer conjugate.
\end{cor}

\begin{rem}
It follows that if $A\times_\al \bZ_+$ and $B\times_\be \bZ_+$ are isometrically
isomorphic, then $\al$ is an automorphism if and only if $\be$ is an automorphism;
and $\al$ is surjective if and only if $\be$ is surjective.
\end{rem}

\subsection{The injective case} \label{Ss:injective}

In this subsection, we assume that $\al$ is injective.
In this case, there is a standard method to imbed $(A,\al)$ into an automorphic system.

Let $(A,\al)$ be an injective dynamical system.
Define the direct limit dynamical system $(A_\infty,\al_\infty)$ by
\begin{align*}
 \xymatrix{
  A \ar[r]^{\al} \ar[d]^\al &
  A \ar[r]^{\al} \ar[d]^\al &
  A \ar[r]^{\al} \ar[d]^\al &
  \cdots \ar[r] &
  A_\infty \ar[d]^{\al_\infty} \\
  A \ar[r]^\al &
  A \ar[r]^\al &
  A \ar[r]^\al &
  \cdots \ar[r] &
  A_\infty
 }
\end{align*}
Then $A$ is contained injectively in $A_\infty$, say via a map $\imath$.
The limit map $\al_\infty$ is an automorphism.

Let the unitary generator of the crossed product $A_\infty \rtimes_{\al_\infty} \bZ$
be $\tilde\Bv$.
It is shown in \cite{KakKat10, Kak11} that the semicrossed product
$A\times_\al \bZ_+$ is unitally completely isometrically isomorphic to the
subalgebra generated by $i(A)$ and $\tilde\Bv$.

Note that $ad_{\tilde\Bv^*}$ defines an injective endomorphism of $A \times_\al \bZ_+$.
Therefore we can construct the direct sequence
\begin{align*}
\xymatrix{ A \times_\al \bZ_+ \ar[r]^{\ad_{\tilde\Bv^*}} &
 A \times_{\al} \bZ_+ \ar[r]^{\ad_{\tilde\Bv^*}} &
 A \times_{\al} \bZ_+ \ar[r]^{\ad_{\tilde\Bv^*}} &
 \cdots
 }
\end{align*}
Observe that $ad_{\tilde\Bv^*}$ induces (concretely) the injective endomorphism $\al$.
Also, $ad_{\tilde\Bv^*}(\tilde\Bv)=\tilde\Bv$.
Hence it follows that the direct limit of the above sequence is
$A_\infty \times_{\al_\infty} \bZ_+$.
Let $\tilde\imath$ denote the injection of $\scp A \al$ into
$A_\infty \times_{\al_\infty} \bZ_+$ extending $\imath$.
Observe that $\tilde\imath(\Bv) = \tilde\Bv$.

Now assume that $\phi$ is an isometric isomorphism of $A\times_\al \bZ_+$
onto $B\times_\be \bZ_+$.
Then $\be$ is also injective by Corollary~\ref{C:ker}.
So we can perform the same direct limit construction for $(B,\be)$
to obtain an automorphic limit system $(B_\infty,\be_\infty)$.
We denote the unitary generator of the crossed product
$B_\infty \rtimes_{\be_\infty} \bZ$ by $\tilde\Bw$.
Let $\jmath$ be the injection of $B$ into $B_\infty$;
and let $\tilde\jmath$ denote the injection of $\scp B \be$ into $\scp{B_\infty}{\be_\infty}$.

Construct the following diagram
\begin{align*}
 \xymatrix{
 A \times_\al \bZ_+ \ar[r]^{\ad_{\tilde\Bv^*}} \ar[d]^{\phi_0} &
 A \times_{\al} \bZ_+ \ar[r]^{\ad_{\tilde\Bv^*}} \ar[d]^{\phi_1} &
 A \times_{\al} \bZ_+ \ar[r]^{\ad_{\tilde\Bv^*}} \ar[d]^{\phi_2} &
 \cdots \ar[r] &
 \scp{A_\infty}{\al_\infty} \ar[d]^{\phi_\infty}\\
 B\times_\be \bZ_+ \ar[r]^{\ad_{\tilde\Bw^*}} &
 B \times_{\be} \bZ_+ \ar[r]^{\ad_{\tilde\Bw^*}} &
 B \times_{\be} \bZ_+ \ar[r]^{\ad_{\tilde\Bw^*}} &
 \cdots  \ar[r] &
 \scp{B_\infty}{\be_\infty}
 }
\end{align*}
where $\phi_n:=ad_{\tilde\Bw^*} \circ \phi_{n-1} \circ ad_{\tilde\Bv}$ for $n\in \bZ_+$.
The diagram clearly commutes.
Therefore the limit map $\phi_\infty$ is a completely isometric isomorphism
of $\scp{A_\infty}{\al_\infty}$ onto $\scp{B_\infty}{\be_\infty} $.

Moreover, one can see that the isomorphism $\ga_\infty$ of $A_\infty$ onto $B_\infty$
extends $\ga$.
Also the generator $\tilde\Bv = \tilde\imath(\Bv)$ and $\tilde\jmath(\Bw) = \tilde\Bw$.
Therefore
\[
 \phi_\infty(\tilde\Bv) = \tilde\jmath( \phi(\Bw) ) =
 \jmath(b_0) + \tilde\Bw \jmath(b_1) + \tilde\Bw^2 \jmath(Y) .
\]

This yields a proof of our theorem in the injective case.

\begin{thm} \label{T:inj}
Let $(A,\al)$ and $(B,\be)$ be unital C*-algebra dynamical systems.
Suppose that $A\times_\al \bZ_+$ and $B\times_\be \bZ_+$ are isometrically isomorphic.
If $\al$ is injective, then $(A,\al)$ and $(B,\be)$ are outer conjugate.
\end{thm}

\begin{proof}
The preceding argument shows that the two automorphic systems
$(A_\infty,\al_\infty)$ and $(B_\infty,\be_\infty)$ have isometrically isomorphic
semicrossed products.
Therefore by Corollary~\ref{C:auto}, these two dynamical systems are outer conjugate.
Indeed, the proof shows more precisely that
\[ \be_\infty(b) = (\ad_w \ga_\infty\al_\infty\ga_\infty^{-1}) (b) \qforal b \in B_\infty \]
where $w$ is the unitary in the polar decomposition of $j(b_1)$.
In particular, $\jmath(b_1)$ is invertible in $B_\infty$,
and hence is invertible in $\jmath(B)$.
Therefore $w$ belongs to $\jmath(B)$.
Restricting to $b$ in $B$ yields the desired outer conjugacy of $(A,\al)$ and $(B,\be)$.
\end{proof}

\section{Main Theorem~\ref{T:main2}} \label{S:main2}

\subsection{Automatic Outer Conjugacy}

In this subsection we show how outer conjugacy is obtained in a number of cases, other than those described in Section~\ref{S:main}, because of \emph{a} structure of the C*-algebras.\\

First of all we can obtain the following strengthening of the main result of \cite{DKsimple}, which required $\al$ and $\be$ to be automorphisms.

\begin{cor} \label{C:simple}
Let $A$ be a simple C*-algebra; and
let $(A,\al)$ and $(B,\be)$ be unital C*-algebra dynamical systems.
Suppose that $A\times_\al \bZ_+$ and $B\times_\be \bZ_+$ are isometrically isomorphic.
Then $(A,\al)$ and $(B,\be)$ are outer conjugate.
\end{cor}

\begin{proof}
Since $A$ is simple, $\al$ is injective and Theorem~\ref{T:inj} applies.
\end{proof}

We can improve this to the case of trivial centre.

\begin{cor} \label{C:trivial_centre}
If the centre of $A$ is trivial $(Z(A) = \bC 1)$, then
the two systems are outer conjugate.
\end{cor}

\begin{proof}
By Lemma~\ref{L:identities}(ii), $b_0b_0^* \in Z(B)$.
If $b_0=0$, then $b_1$ is unitary by Lemma~\ref{L:a0=0}(i).
If $b_0\ne0$, then $b_0b_0^*$ is a non-zero scalar in
$\ann(R_\be)$ by Lemma~\ref{L:annihilator}.
Hence $R_\be = \{0\}$.  Therefore $\be$ is injective.
So $b_1$ is invertible by Theorem~\ref{T:inj}.
In either case, the two systems are outer conjugate by the Key Lemma~\ref{L:key}.
\end{proof}

When $A$ is commutative,
Lemma~\ref{L:rightinv} shows that $b_1$ is right invertible and thus invertible.
Thus Lemma~\ref{L:key} applies.
A stronger result is known in the commutative case \cite{DK},
namely: if  two semicrossed products of commutative C*-algebras
are \textit{algebraically isomorphic}, then the dynamical systems are conjugate.
We record the isometric case as another corollary.

\begin{prop} \label{P:abelian}
Let $A$ be an abelian C*-algebra; and
let $(A,\al)$ and $(B,\be)$ be unital C*-algebra dynamical systems.
Suppose that $A\times_\al \bZ_+$ and $B\times_\be \bZ_+$ are isometrically isomorphic.
Then $(A,\al)$ and $(B,\be)$ are conjugate.
\end{prop}

Indeed, when $A$ is any finite C*-algebra (i.e.\ no proper isometries),
right invertibility implies invertibility.
So we have:

\begin{prop} \label{P:finite}
Let $A$ be a finite unital C*-algebra;
and let $(A,\al)$ and $(B,\be)$ be unital C*-algebra dynamical systems.
Suppose that the semicrossed products
$A\times_\al \bZ_+$ and $B\times_\be \bZ_+$ are isometrically isomorphic.
Then $(A,\al)$ and $(B,\be)$ are outer conjugate.
\end{prop}

Since  it suffices that $b_1$ lie in a finite subalgebra of $B$ to reach this conclusion,
we obtain two  results of this type.

\begin{prop} \label{P:centre}
Let $(A,\al)$ and $(B,\be)$ be unital C*-algebra dynamical systems.
Suppose that $A\times_\al \bZ_+$ and $B\times_\be \bZ_+$ are isometrically isomorphic,
and that $b_1 = E_1 \phi(\Bv)$ belongs to the centre $Z(B)$.
Then $(A,\al)$ and $(B,\be)$ are outer conjugate.
\end{prop}

\begin{prop} \label{P:finite_commutant}
Let $(A,\al)$ and $(B,\be)$ be unital C*-algebra dynamical systems,
and suppose that $\al(A)'$, the commutant of $\al(A)$ in $A$, is a finite C*-algebra.
Then if $A\times_\al \bZ_+$ and $B\times_\be \bZ_+$ are isometrically isomorphic,
it follows that $(A,\al)$ and $(B,\be)$ are outer conjugate.
\end{prop}

\begin{proof}
This result relies on Lemma~\ref{L:identities}(iii) that $a_1 \ga^{-1}(b_1) \in \al(A)'$.
By Lemma~\ref{L:rightinv}, we know that $a_1 \ga^{-1}(b_1)$ is right invertible.
By hypothesis, it lies in a finite algebra, and therefore it is invertible. So $b_1$
is invertible.  Hence the two systems are outer conjugate by the Key Lemma~\ref{L:key}.
\end{proof}

An immediate consequence of Lemma~\ref{L:a0=0} and the Key Lemma~\ref{L:key}
is the following.

\begin{prop} \label{P:b0=0}
Let $(A,\al)$ and $(B,\be)$ be unital C*-algebra dynamical systems.
Suppose that $A\times_\al \bZ_+$ and $B\times_\be \bZ_+$ are isometrically isomorphic.
If $b_0 := E_0\phi(\Bv) = 0$ or if $E_n\phi(\Bv)=0$ for all $n\ge2$ $($i.e.\ $Y=0)$, 
then $(A,\al)$ and $(B,\be)$ are outer conjugate.
\end{prop}

\subsection{Analysis of the Kernels}

We complete the proof of Theorem~\ref{T:main2}. First we need a general result about quotients of semicrossed products.

\begin{lem} \label{L:quotient}
Let $I$ be an ideal of $A$ such that $\al(I) \subset I$,
and let $\dot\al$ denote the induced endomorphism of $A/I$.
Let $I_\al$ denote the ideal of $\scp A \al$ generated by $I$,
namely
\[ I_\al = \{ X \in \scp A \al : E_n(X) \in I \FORAL n \ge 0 \} .\]
Then $\scp {A/I}{\dot\al}$ is completely isometrically isomorphic to
$\scp A \al/I_\al$.
\end{lem}

\begin{proof}
Let $q$ be the quotient map of $A$ onto $A/I$; and
let $\dot\Bv$ denote the generator of $\scp {A/I}{\dot\al}$.
Every covariant representation $(\pi,V)$ of $(A/I,\dot\al)$ gives
rise to a covariant representation $(\pi q,V)$ of $(A,\al)$.
Therefore there is a canonical completely contractive map $\dot q = q \times \dot\Bv$
of $\scp A \al$ into $\scp {A/I}{\dot\al}$ such that $\dot q|_A = q$
and $\dot q(\Bv) = \dot\Bv$.

It is easy to verify that $\ker \dot q = I_\al$.
Therefore there is an injective completely contractive map $\si$ of
$\scp A \al/I_\al$ into $\scp {A/R_\al}{\dot\al}$ such that
$\si|_{A/R_\al} = \id$ and $\si(\tilde\Bv) = \dot\Bv$,
where $\tilde\Bv = \Bv + I_\al$.

Conversely, suppose that $\scp A \al/I_\al$ is represented
completely isometrically on a Hilbert space $\H$ by $\hat\pi$.
It is completely determined by the restriction to $A/I$,
which is a $*$-isomorphism $\pi$, and the contraction $V = \hat\pi(\dot\Bv)$.
It is evident that $(\pi,V)$ yields a covariant representation of $(A/I,\dot\al)$.
Therefore by the universal property, $\pi\times V$ is a completely contractive map of
$\scp {A/I}{\dot\al}$ into $\hat\pi(\scp A \al/I_\al)$.
When composed with $\hat\pi^{-1}$, one obtains a completely contractive map $\tau$
of  $\scp {A/I}{\dot\al}$ into $\scp A \al/I_\al$ which restricts to the identity on
$A/R_\al$ and takes $\dot\Bv$ to $\tilde\Bv$.
Clearly $\tau = \si^{-1}$. Therefore $\si$ is a completely isometric isomorphism.
\end{proof}

Now we identify how an isometric isomorphism passes to quotients.

\begin{lem} \label{L:induced}
Suppose that $\phi$ is a completely isometric isomorphism of $\scp A \al$
onto $\scp B \be$.
Let $I$ be an ideal of $A$ and $J =\ga(I)$ an ideal in $B$ such that
$\al(I) \subset I$ and $\be(J) \subset J$.
Let $\dot\al$ denote the induced endomorphism of $A/I$, and
let $I_\al$ be the ideal of $\scp A \al$ generated by $I$.
Similarly, define $\dot\be$ and $J_\be$.
Then $I_\al$ and $I_\be$ are completely isometrically isomorphic
via the restriction of $\phi$,
and $\phi$ induces a completely isometric map $\dot\phi$ of
$\scp {A/I}{\dot\al}$ onto $\scp {B/J}{\dot\be}$.
\end{lem}

\begin{proof}
First observe that $\phi(I_\al) = J_\be$.
Indeed, a dense subset of $I_\al$ is given by polynomials
$p = \sum_{k=0}^n \Bv^k r_k$ where $r_k \in I$.
So $\phi(p) = \sum_{k=0}^n \phi(\Bv)^k \ga(r_k)$ lies in $J_\be$.
Similarly, $\phi^{-1}(J_\be) \subset I_\al$.
Therefore we get the desired equality.

Combining this with the previous lemma, we obtain the following diagram:
\[
 \xymatrix{
 \scp A \al \ar[r]^{\phi} \ar[d]^{\tilde q} & \scp B \be \ar[d]^{\tilde q} \\
 \scp A \al/I_\al \ar[r]^{\tilde\phi} \ar[d]^{\simeq} &
 \scp B \be/J_\be  \ar[d]^{\simeq} \\
 \scp {A/I}{\dot\al} \ar[r]^{\dot\phi} & \scp {B/J}{\dot\be}
 }
\]
Since the vertical maps labelled $\tilde q$ are complete quotient maps,
and $\phi(I_\al)=J_\be$, the map $\tilde\phi$ is also completely isometric.
Combining this with the complete isometries of Lemma~\ref{L:quotient}
yields that  $\dot\phi$ is a complete isometry.
\end{proof}

Define $\dot\al \in \End(A/R_\al)$ by $\dot{\al}(a+R_\al)=\al(a)+R_\al$.
It is easy to see that  the dynamical system $(A/R_\al, \dot\al)$ is injective.
Let $\R_\al$ denote the ideal of $\scp A \al$ generated by $R_\al$, i.e.
\[ \R_\al = \{ X \in \scp A \al : E_n(X) \in R_\al \FORAL n \ge 0 \} .\]
Recall that $\phi(\Bv) = b_0 + \Bw b_1 + \Bw^2 Y$.
Let $q_\infty$ be the quotient map $q_\infty:B\to B/R_\be$.
The previous two lemmas combine to yield the first two parts of this proposition.

\begin{prop} \label{P:quotientRal} With the above notation,

\textup{(i)}
$\scp {A/R_\al}{\dot\al}$ is completely isometrically isomorphic to $\scp A \al/\R_\al$.

\textup{(ii)}
The isomorphism $\phi$ induces a completely isometric isomorphism $\dot\phi$ of
$\scp {A/R_\al}{\dot\al}$ onto $\scp {B/R_\be}{\dot\be}$.

\textup{(iii)}
The systems $(A/R_\al,\dot\al)$ and $(B/R_\be,\dot\be)$ are outer conjugate.
In particular, $q_\infty(b_1)$ is invertible in $B/R_\be$.
\end{prop}

\begin{proof}
The first statement is immediate from Lemma~\ref{L:quotient}.
The second statement follows from Corollary~\ref{C:kernels} and Lemma~\ref{L:induced}.

Observe that $\al^{-1}(R_\al) \subset R_\al$ because $x \in \al^{-1}(R_\al)$
implies that $\al(x) \in R_\al$. Hence $\lim_n \al^{n+1}(x) = 0$,
showing that $x$ is in $R_\al$.
Therefore the dynamical system $(A/R_\al,\dot\al)$ is injective.
It now follows from Theorem~\ref{T:inj} that
$(A/R_\al,\dot\al)$ and $(B/R_\be,\dot\be)$ are outer conjugate.
Moreover the proof shows that the coefficient of $\dot\Bw$ in $\dot\phi(\dot\Bv)$,
namely $q_\infty(b_1)$, is invertible.
\end{proof}

Now using the fact that $R_\be$ is the union of an increasing sequence of ideals,
we improve the previous result to a quotient by the kernel of a finite power of $\al$.
Let $q_n$ be the quotient map $q_n:B\to B/\ker(\be^n)$ for $n \ge 1$.

\begin{cor} \label{C:quotient al^n0}
Suppose that $\phi$ is a completely isometric isomorphism of $\scp A \al$
onto $\scp B \be$.
Then there is an integer $n_0$ so that $q_{n_0}(b_i)$ is invertible in
$B/\ker\be^{n_0}$. Therefore
$(A/\ker\al^{n_0},\tilde\al)$ and $(B/\ker\be^{n_0},\tilde\be)$ are outer conjugate.
\end{cor}

\begin{proof}
Since $q_\infty(b_1)$ is invertible and $R_\al$ is the closed union
of $\ker(\al^n)$ for $n \ge 1$, it follows that there is an integer $n_0$ so that
$b_1 + \ker(\be^{n_0})$ is invertible in $B/\ker\be^{n_0}$.
Therefore $(A/\ker\al^{n_0},\tilde\al)$ and $(B/\ker\be^{n_0},\tilde\be)$ are
outer conjugate by Lemma~\ref{L:induced} and the Key Lemma~\ref{L:key}.
\end{proof}

For $k \ge 0$, let
$A_k = A/\ker(\al^k)$ and $J_k = \ker(\al^{k+1})/\ker(\al^k)$.
Choose a faithful non-degenerate representation $\pi_k$ of $J_k$.
This induces a representation of $A_k$ on the same space, which we also call $\pi_k$.
Also let $\sigma_n$ denote a faithful non-degenerate representation of $A_n$.
Then
\[ \rho = \sum_{k=0}^{n_0-1} \strut^\oplus \pi_k q_k \oplus \sigma_{n_0}q_{n_0} \]
yields a representation of $A$.
It follows from basic C*-algebra theory that $\rho$ is faithful.

\begin{lem} \label{L:al(Ral)}
Consider the following properties:
\begin{enumerate}
\renewcommand{\labelenumi}{(\arabic{enumi}) }
\item $\al(R_\al) = R_\al$.
\item $\al(\ker(\al^{k+1})) = \ker(\al^k)$ for all $k \ge 0$.
\item $\al(\ann(R_\al)) \subseteq \ann(R_\al)$.
\item $\al^k(a_0) \in \ann(R_\al)$ for all $k \ge 0$.
\item $\al(a_0) \in \ann(R_\al)$.
\item $\al^k(a_0) \in \ann(\ker(\al))$ for all $k \ge 0$.
\item $\al(a_0) \ker(\al^{k+1}) \subset \ker(\al^k)$ for all $k \ge 0$.
\item $\pi_k q_k(a_0) = \pi_k q_k \al (a_0) = 0$ for all $k \ge 0$.
\end{enumerate}
Then
\[
 (1) \iff (2) \implies (3)  \implies (4)
 \begin{matrix}
 \begin{turn}{25}$\implies$\end{turn} \\[-1ex]
 \begin{turn}{-25}$\implies$\end{turn}
 \end{matrix}
 \begin{matrix} (5) \\[2.6ex] (6)\end{matrix}
 \begin{matrix}
 \begin{turn}{-25}$\implies$\end{turn} \\[2ex]
 \begin{turn}{25}$\implies$\end{turn}
 \end{matrix}
 (7) \iff (8) .
\]
\end{lem}

\begin{proof}
Clearly $\al(\ker(\al^{k+1}))$ is contained in $\ker(\al^k)$.
If they are not equal, take any $x \in \ker(\al^k)$ which is
not in $\al(\ker(\al^{k+1}))$. Then $x$ is not in the range of $\al$.
So $\al(R_\al)$ does not contain $x$. Therefore (1) implies (2).
Conversely, it is clear that if (2) holds, then $\al(R_\al)$ contains
$\ker(\al^k)$ for all $k \ge 0$. Since the image is a C*-algebra, it is closed
and thus is all of $R_\al$. So (2) implies (1).

Clearly (1) implies (3) because then
\[ \al(\ann(R_\al)) \subseteq \ann(\al(R_\al)) = \ann(R_\al) . \]

By Lemma~\ref{L:annihilator}, $a_0 \in \ann(R_\al)$.
Thus (3) implies that $\al^k(a_0)$ lies in $\ann(\al^k(R_\al))$ for all $k \ge 0$. So (4) follows.
Clearly (4) implies (5) and (6); and (5) trivially implies (7).

Assume (6), and suppose that $x \in \ker(\al^{k+1})$.
Then $\al^k(x) \in \ker(\al)$. So
\[ \al^k( \al(a_0)x) ) = \al^{k+1}(a_0) \al^k(x) = 0 .\]
Therefore $\al(a_0)x \in \ker(\al^k)$; whence  (6) implies (7).

Since $a_0 \in \ann(R_\al)$, we always have $\pi_k q_k(a_0) = 0$.
Assume that (7) holds.
The representation $\pi_k$ is non-degenerate on $J_k$.
Thus $\pi_k(J_k) = \pi_k q_k (\ker(\al^{k+1}) )$ has dense range in $\H_{\pi_k}$.
By (7), $\pi_k q_k \al (a_0 x) = 0$ for all $x \in \ker(\al^{k+1})$.
So $\pi_k q_k \al (a_0) = 0$. Hence (7) implies (8).

Conversely, if (8) holds, then $\pi_k q_k (\al (a_0) x) = 0$ for all $x \in \ker(\al^{k+1})$.
Since $\pi_k$ is faithful on $J_k$, it follows that $q_k (\al (a_0) x) = 0$; namely,
$\al(a_0) x$ lies in $\ker(\al^k)$. So (7) follows.
\end{proof}

\begin{lem} \label{L:b1 inv mod k}
Assume that $\al(a_0) \ker(\al^{k+1}) \subset \ker(\al^k)$ for all $k \ge 0$.
Then $\pi_k q_k (a_1)$ is invertible for all $k \ge 0$.
\end{lem}

\begin{proof}
By Lemma~\ref{L:al(Ral)}, we have that $\pi_k q_k(a_0) = \pi_k q_k \al (a_0) = 0$
for all $k \ge 0$. Lemma~\ref{L:a0=0} shows that
\[ 1 = a_1 \ga^{-1}(b_1) + \al(a_0)a_1c_0 + a_1a_0c_0 + \al(a_0)^2c_1 .\]
Apply $ \pi_k q_k$ to obtain
\[ \pi_k q_k (1) = \pi_k q_k (a_1 \ga^{-1}(b_1)) .\]
Thus $\pi_k q_k \ga^{-1}(b_1)$ is left invertible, and so invertible by
Lemma~\ref{L:rightinv}. Now as in Lemma~\ref{L:a0=0}, it follows that
$\pi_k q_k \ga^{-1}(b_1)$ and $\pi_k q_k(a_1)$ are unitary.
\end{proof}

We obtain the following partial result in the non-injective case.

\begin{prop}\label{P:b0orbit}
Suppose that $\phi$ is an isometric isomorphism of the semicrossed product
$A\times_\al \bZ_+$  onto $B\times_\be \bZ_+$.
If any of the conditions of Lemma~$\ref{L:al(Ral)}$ holds,
particularly
\begin{itemize}
\item[(1)] $\al(R_\al) = R_\al$\quad or
\quad$(3)$\ $\al(\ann(R_\al)) \subseteq \ann(R_\al)$,
\end{itemize}
then $(A,\al)$ and $(B,\be)$ are outer conjugate.
\end{prop}

\begin{proof}
By the previous two lemmas, we obtain that $\pi_k q_k (a_1)$ is invertible for all $k \ge 0$.
By Corollary~\ref{C:quotient al^n0}, we also see that there is an $n_0$ so that
$\si_{n_0} q_{n_0}(a_1)$ is invertible. Therefore $\rho(a_1)$ is invertible.
This is a faithful representation, so $a_1$ is invertible.
The result follows from the Key Lemma \ref{L:key}.
\end{proof}

As one last result, we make an improvement to Proposition~\ref{P:finite_commutant}.
Let $\rho$ be a faithful representation of $R_\infty$ on $\H$.
This induces a representation of $A$, which we also call $\rho$, 
which factors through 
\[ A \to A/\ann(R_\infty) \to M(R_\infty) ,\]
where $M(R_\infty)$ is the multiplier algebra of $R_\infty$.
Since $\rho$ extends to a faithful representation of $M(R_\infty)$,
we may consider $\rho(A)$ as a subalgebra of $M(R_\infty)$.

\begin{prop} \label{P:rho_finite_commutant}
Suppose that $\phi$ is an isometric isomorphism of the semicrossed product
$A\times_\al \bZ_+$  onto $B\times_\be \bZ_+$.
If $\rho(\al(A)')$, the image of the commutant of $\al(A)$
in $M(R_\infty)$, is a finite C*-algebra,
then $(A,\al)$ and $(B,\be)$ are outer conjugate. 
\end{prop}

\begin{proof}
The argument again is to show that $a_1$ is invertible.
A faithful representation of $A$ is obtained from $\pi q_\infty \oplus \rho$,
where $\pi$ is a faithful representation of $A/R_\infty$.
By Proposition~\ref{P:quotientRal}, $\pi q_\infty(a_1)$ is invertible.
Now $\rho(a_1\ga^{-1}(b_1))$ belongs to $\rho(\al(A)')$ by Lemma~\ref{L:identities}(iii).
This element is right invertible by Lemma~\ref{L:rightinv}.
By hypothesis, $\rho(\al(A)')$ is finite. Therefore $\rho(a_1\ga^{-1}(b_1))$ is invertible.
Hence $\rho(\ga^{-1}(b_1))$ is invertible, and so $\rho(a_1)$ is also invertible.
Combining these two facts, we obtain that $a_1$ is invertible.
Thus the Key Lemma~\ref{L:key} applies to show that 
$(A,\al)$ and $(B,\be)$ are outer conjugate. 
\end{proof}



\begin{thebibliography}{99}

\bibitem{Arv} W. Arveson,
\textit{Operator algebras and measure preserving automorphisms},
Acta Math.\ \textbf{118}, (1967), 95--109.

\bibitem{ArvJ} W. Arveson and K. Josephson,
\textit{Operator algebras and measure preserving automorphisms II},
J. Funct.\ Anal.\ \textbf{4} (1969), 100--134.

\bibitem{DK} K. Davidson and E. Katsoulis,
\textit{Isomorphisms between topological conjugacy algebras},
J. Reine Angew.\ Math.\ (Crelle),  \textbf{621} (2008), 29--51.

\bibitem{DKsimple} K. Davidson and E. Katsoulis,
\textit{Semicrossed Products of Simple C*-algebras},
Math. Ann. \textbf{342} (2008), 515--525.

\bibitem{DKsurvey} K. Davidson and E. Katsoulis,
{\it Nonself-adjoint operator algebras for dynamical systems},
in {\sf Operator Structures and Dynamical Systems},
M. de Jeu, S. Silvestrov, C. Skau, and J. Tomiyama, eds.,
Contemporary Math. \textbf{503} (2009), pp. 39--51.

\bibitem{Exel03} R. Exel,
\textit{A new look at the crossed-product of a C*-algebra by an endomorphism},
Ergodic Theory Dynam. Systems \textbf{23}(6) (2003), 1733--1750.

\bibitem{HadH}  D. Hadwin and T. Hoover,
\textit{Operator algebras and the conjugacy of transformations},
J. Funct.\ Anal.\ \textbf{77} (1988), 112--122.

\bibitem{HoaPar66}  H. Hoare and W. Parry,
\textit{Affine transformations with quasi-discrete spectrum. {I}},
J. London Math. Soc. \textbf{41} (1966), 88--96.

\bibitem{Kak11} E. Kakariadis,
\textit{Semicrossed products of C*-algebras and their C*-envelopes},
preprint (arXiv.org: 1102.2252v2).

\bibitem{KakKat10} E. Kakariadis and E. Katsoulis,
\textit{Semicrossed Products of Operator Algebras and their C*-envelopes},
J. Funct.\ Anal.\ \textbf{262}(7) (2012), 3108--3124.

\bibitem{KakKat11} E. Kakariadis and E. Katsoulis,
\textit{Contributions to the theory of C*-Correspondences with
applications to multivariable dynamics},
Trans.\ Amer.\ Math.\ Soc., to appear.

\bibitem{Kish81}  A. Kishimoto,
\textit{Outer automorphisms and reduced crossed products of simple C*-algebras},
Comm.\ Math.\ Phys.\ \textbf{81} (1981), 429--435.

\bibitem{MS00} P.S. Muhly, B. Solel,
\textit{On the Morita Equivalence of Tensor Algebras},
Proc.\ London Math.\ Soc. \textbf{3} (2000), 113---168.

\bibitem{MS06} P. Muhly and B. Solel,
\textit{Extensions and dilations for C*-dynamical systems},
Operator theory, operator algebras and applications,
Contemp.\ Math.\ \textbf{414} (2006), 375--381.

\bibitem{Pas} W. Paschke,
\textit{The crossed product of a C*-algebra by an endomorphism}
Proc. Amer. Math. Soc. \textbf{80} (1980), 113--118.

\bibitem{Peters} J. Peters,
\textit{Semicrossed products of C*-algebras},
J. Funct.\ Anal.\ \textbf{59} (1984), 498--534.

\bibitem{Stac} P. Stacey,
\textit{Crossed products of C*-algebras by endomorphisms},
J. Austr. Math. Soc., Ser. A \textbf{56} (1993), 204--212.


\end{thebibliography}
\end{document}